\documentclass[11pt]{article}
\usepackage{latexsym,amsmath,stackrel,color,amsthm,amssymb,epsfig,graphicx,mathrsfs}
\usepackage{graphicx}
\usepackage{amssymb}
\usepackage[left=1in,top=1in,right=1in,bottom=1in]{geometry}
\usepackage[linktocpage=true]{hyperref}
\usepackage{setspace}
\usepackage{amssymb, amsmath, amsthm, graphicx,mathrsfs}

\newtheorem{thm}{Theorem}
\newtheorem{cor}[thm]{Corollary}

\newtheorem{definition}{Definition}

\newtheorem{lem}[thm]{Lemma}

\newtheorem{prop}[thm]{Proposition}

\newcommand{\block}{\mathcal{B}}
\newcommand{\E}{\mathbb{E}}
\newcommand{\R}{\mathbb{R}}

\newcommand{\I}{\mathcal{I}}
\renewcommand{\l}{\left}
\renewcommand{\r}{\right}
\renewcommand{\P}{\mathbb{P}}
\newcommand{\pushright}[1]{\ifmeasuring@#1\else\omit\hfill$\displaystyle#1$\fi\ignorespaces}

\newfont{\footsc}{cmcsc10 at 8truept}
\newfont{\footbf}{cmbx10 at 8truept}
\newfont{\bigrm}{cmr12 scaled\magstep4}
\newfont{\footrm}{cmr10 at 10truept}

\begin{document}

\title{Randomized greedy algorithm for independent sets in regular uniform hypergraphs with large girth}

\medskip

\author{
Jiaxi Nie\footnote{E-mail: jin019@ucsd.edu} \qquad \qquad Jacques Verstra\" ete\footnote{E-mail: jacques@ucsd.edu} \\
\\
Department of Mathematics \\
University of California, San Diego \\
9500 Gilman Drive \\
La Jolla CA 92093-0112.
}

\maketitle

\begin{abstract}
In this paper, we consider a randomized greedy algorithm for independent sets in $r$-uniform $d$-regular hypergraphs $G$ on $n$ vertices with girth $g$. By analyzing the expected size of the independent sets generated by this algorithm, we show that $\alpha(G)\geq (f(d,r)-\epsilon(g,d,r))n$, where $\epsilon(g,d,r)$ converges to $0$ as $g\rightarrow\infty$ for fixed $d$ and $r$, and $f(d,r)$ is determined by a differential equation. This extends earlier results of Gamarnik and Goldberg for graphs~\cite{GG}. We also prove that when applying this algorithm to uniform linear hypergraphs with bounded degree, the size of the independent sets generated by this algorithm concentrate around the mean asymptotically almost surely.
\end{abstract}


\section{Introduction}

A {\em hypergraph} is a pair $(V,E)$ where $V$ is a set and $E$ is a family of nonempty subsets of $V$. The $x\in V$ are called vertices and the $e\in E$ are called edges. We use the notations $v(G)=|V(G)|$, $e(G)=|E(G)|$. A hypergraph is called {\em $r$-uniform} if all edges have size $r$. A {\em linear hypergraph} is a hypergraph $(V,E)$ such that for any distinct edges $e,f \in E$, $|e \cap f| \leq 1$. The {\em degree} of a vertex $v$, denoted by $d(v)$, is the number of edges that contains it. A hypergraph is {\em $d$-regular} if all vertices have degree $d$. An {\em independent set} of a hypergraph $G$ is a subset of $V(G)$ which does not contain any edge of $G$. The maximum size of an independent set in $G$ is called the {\em independence number} of $G$, denoted $\alpha(G)$. 

In this paper, we study a natural randomized greedy algorithm for finding independent sets in hypergraphs. The algorithm  iteratively selects a vertex uniformly randomly from all remaining vertices of the hypergraph and adds it to the independent set so far, and then deletes all remaining vertices that form an edge with the set of selected vertices, and repeat until no vertices remain. The independent set generated by this algorithm for a hypergraph $G$ is denoted $\I(G)$.

\subsection{Independent sets in graphs}

Tur\'an's Theorem~\cite{T} shows that an $n$-vertex graph with average degree $d$ has independence number $\alpha(G)\geq n/(d+1)$, with equality only for a disjoint union of cliques $K_{d + 1}$. 
For triangle-free graphs $G$, Ajtai, Koml\'os and Szemer\'edi~\cite{AKS} improved this bound by a factor of order $\log d$, and Shearer~\cite{She83} gave a further improvement:

\begin{thm} 
Let $G$ be an $n$-vertex graph triangle-free of average degree $d \geq 2$. Then 
\begin{equation}\label{She}
    \alpha(G)\ge\frac{d\log d-d+1}{(d-1)^2}\cdot n.
\end{equation}
\end{thm}

The {\em girth} of a graph containing a cycle is the length of a shortest cycle in the graph. For graphs with high girth, Shearer~\cite{She} improved (\ref{She}), and Lauer and Wormald~\cite{LW} showed that 
there exist a function $\delta = \delta(g)$ such that $lim_{g \rightarrow \infty} \delta(g) = 0$ and if $G$ is a $d$-regular graph of girth $g$, then 
\begin{equation}
\alpha(G) \geq \frac{1}{2}(1-(d-1)^{-\frac{2}{d-2}})n  -  \delta n.
\end{equation}
By analyzing the performance of the greedy algorithm, Gamarnik and Goldberg~\cite{GG} prove the same bound, with an explicit form for $\delta$. It is convenient to let
 \begin{equation}
 \epsilon=\epsilon(d,g)= \frac{d(d-1)^{\lfloor\frac{g-3}{2}\rfloor}}{(\lfloor\frac{g-1}{2}\rfloor)!}.
 \end{equation} 
 Note that for each fixed $d$, $\epsilon(d,g) \rightarrow 0$ as $g \rightarrow \infty$.

\begin{thm}\label{GG}
	Let integers $d\geq3$ and $g\geq4$, and let $G$ be a $d$-regular graph on $n$ vertices with girth $g$, and let $\mathcal{I}$ be the independent set generated by the greedy algorithm. Then
	\begin{equation}
	    \Bigl(\frac{1-(d-1)^{-2/(d-2)}}{2}-\epsilon\Bigr)n\le\E[|\mathcal{I}|]\le\Bigl(\frac{1-(d-1)^{-2/(d-2)}}{2}+\epsilon\Bigr)n,
	\end{equation}
\end{thm}

The bounds are effective when $d$ is fixed and $g$ is large and, in particular, Theorem \ref{GG} shows
\begin{equation}
    \alpha(G)\geq\Bigl(\frac{1-(d-1)^{-2/(d-2)}}{2}-\epsilon\Bigr)n.
\end{equation}
We also observe that when $g$ is sufficiently large relative to $d$, this bound agrees with (\ref{She}) asymptotically as $d \rightarrow \infty$, since
$$
(d-1)^{-2/(d-2)}=\exp\l(-\frac{2\log (d-1)}{d-2}\r)=1-\frac{2\log(d-1)}{d-2}(1+o_d(1)),
$$
where $o_d(1)$ here represents a function of $d$ that converges to zero as $d\rightarrow\infty$. When we say a function $f(x)$ is asymptotic to $g(x)$ as $x\rightarrow\infty$(which is abbreviated $f\sim g$), it means that $\lim_{x\rightarrow\infty}f(x)/g(x)=1$.

\subsection{Independent sets in hypergraphs}

For $(r+1)$-uniform hypergraphs with average degree $d$, Caro and Tuza~\cite{CT} showed that
\begin{equation}
    \alpha(G)\ge \frac{d!}{\prod_{i=1}^d(i+\frac{1}{r})}\cdot n.
\end{equation}
The same bound can also be obtained by extending the Caro-Wei~\cite{C}~\cite{W}
bound for independent sets in graphs: taking a random ordering of the vertices of the hypergraph, let $I$ be the set of vertices $v$ such that for every edge $e$ containing $v$, $v$ is not the smallest vertex in $e$. Then it can be shown via elementary combinatorial methods that
\begin{equation}\label{combinatorial}
 \mathbb E[|I|] \geq \sum_{v \in V} \frac{d!}{\prod_{i=1}^d(i+\frac{1}{r})}=\frac{d!}{\prod_{i=1}^d(i+\frac{1}{r})}\cdot n.
 \end{equation}
The same algorithm can be implemented via the following random process, which provide a different (and possibly easier) way to analyze the outcome (see for example in Dutta, Mubayi and Subramanian~\cite{DMS}):
\begin{enumerate}
    \item Equip each vertex with i.i.d.~weight from the uniform distribution on [0,1]. Then with probability 1, all vertices will have distinct weights.
    \item Select all the vertices that are not the smallest-weighted vertex in any edge that contains it. These vertices form an independent set.
\end{enumerate}
If we select vertices in a more careful way -- iteratively select the vertex with largest weight, i.e., select the vertex with largest weight, delete vertices that form an edge with the vertices selected thus far, and repeat -- then this random process will be equivalent to the randomized greedy algorithm.
In any case, a computation shows 
\begin{equation}
 \mathbb E[|I|] \geq n\int_0^1 (1 - x^{r})^ddx
 \end{equation}
which gives (\ref{combinatorial}). These bounds are asymptotic to $\Gamma(1 + \frac{1}{r})nd^{-\frac{1}{r}}$ as $d \rightarrow \infty$, where $\Gamma$ here is the well-known gamma function that extends factorial function to complex numbers. 
In this paper, we consider this algorithm in uniform hypergraphs of large girth. 

\medskip

To define {\em girth} in hypergraph, we first need to define what is a {\em cycle} in hypergraph. There are many different ways to define cycle in hypergraph--see, e.g., a talk by S\'ark\"ozy~\cite{Sar}. Here we chose to work with the {\em Berge-cycle}. For $k\geq3$, a {\em Berge $k$-cycle} is an $r$-uniform hypergraph with $k$ edges $e_1,e_2,\dots,e_k$ such that there exist distinct vertices $v_1,v_2,\dots,v_k$ such that $\{v_k,v_1\}\in e_1,\{v_1,v_2\}\in e_2,\dots,\{v_{k-1},v_k\}\in e_k$. When $k = 2$, this corresponds to $v_1,v_2 \in e_1 \cap e_2$. The {\em girth} of a hypergraph containing a Berge cycle is the smallest $g$ such that the hypergraph contains a Berge $g$-cycle. In particular, the girth of a non-linear hypergraph is 2. Ajtai, Koml\'os, Pintz, Spencer and Szemer\'edi~\cite{AKPSS} established the following lower bound for $(r + 1)$-uniform hypergraphs with girth $g\geq5$, which improves
(\ref{combinatorial}) by a factor of order $(\log d)^{\frac{1}{r}}$. 

\begin{thm}
For integer $r\geq 1$, real number $d$ sufficiently large and integer $n$ sufficiently large, let $G$ be an $n$-vertex $(r+1)$-uniform hypergraphs with average degree $d$ and girth at least $5$, then
\begin{equation}\label{AKPSS}
    \alpha(G)\geq 0.36\cdot10^{-\frac{5}{r}}\l(\frac{\log{d}}{rd}\r)^{\frac{1}{r}}n.
\end{equation}
\end{thm}
Based on this theorem, Duke, Lefmann and R\"odl~\cite{DLR} showed that the same bound(with different constant) holds for linear hypergraphs. 
 
\subsection{Main Theorem}

In this paper, we extend the ideas of Gamarnik and Goldberg~\cite{GG} to hypergraphs. First, it is convenient to define the following: 
Let $u(d,r)$ be the only positive real number that satisfies the following equation:
\begin{equation}\label{ud}
    \sum_{n\ge0}\binom{n+d-2}{d-2}\frac{u(d,r)^{rn+1}}{rn+1}=1.
\end{equation}
Define
\begin{equation}
\epsilon = \epsilon(g,d,r) = \frac{d(d-1)^{\lfloor\frac{g-3}{2}\rfloor}}{r\sum_{k=1}^{\lfloor\frac{g-1}{2}\rfloor}(k+\frac{1}{r})}.
\end{equation}
Our main theorem is as follows: 

\begin{thm}\label{Thm1}
    For any integers $r\geq1$, $d\geq2$ and $g\geq4$, let $G$ be an $(r+1)$-uniform $d$-regular hypergraph with $n$ vertices and girth $g$, let $\mathcal{I}$ be the independent set of $G$ generated by the greedy algorithm.
    Let
    \begin{equation}
    	f(d,r)=u(d,r)-\frac{u(d,r)^{r+1}}{r+1}.
    \end{equation}
Then
    \begin{equation}\label{main}
    	(f(d,r)-\epsilon)n\le\E[|\I|]\le(f(d,r)+\epsilon)n,
    \end{equation}
\end{thm}
In particular, due to the form of the quantity $\epsilon = \epsilon(g,d,r)$, this theorem is effective for fixed $d$ and large $g$, and shows
\begin{equation}
    \alpha(G)\geq(f(d,r)-\epsilon)n.
\end{equation}
For $r=1$, this coincides with Theorem~\ref{GG}. We prove in Appendix A that as $d \rightarrow \infty$, 
\begin{equation}
f(d,r) \sim \Bigl(\frac{\log d}{rd}\Bigr)^{\frac{1}{r}},
\end{equation}
and so if $g$ is large enough relative to $d$, then this slightly improves the constant in (\ref{AKPSS}) asymptotically as $d \rightarrow \infty$.

\medskip

Our second result shows that the size of the independent set generated by the greedy algorithm concentrate around its mean asymptotically almost surely for linear 
hypergraphs with bounded degree (i.e. hypergraphs that are not necessarily regular):

\begin{thm}\label{Thm2}
	For any integers $r\geq1$ and $d\geq2$, let $G$ be an $(r+1)$-uniform linear hypergraph with maximum degree $d$ on $n$ vertices, $\I(G)$ be the independent set generated by the greedy algorithm, then for any positive function $b(n)$ with $b(n)\rightarrow\infty$ as $n\rightarrow\infty$, we have
	$$
	\P[||\I(G)|-\E[|\I(G)|]|>\sqrt{n}b(n)]\rightarrow0,\ as\ n\rightarrow\infty.
	$$
\end{thm}

The rest of this paper is structured as follows. In Section 2, we introduce {\em influence-blocking hypergraphs} and {\em bonus function of hypergraphs}. They are originally notions for graphs from Gamarnik and Goldberg~\cite{GG}, which are generalized to notions for hypergraphs here. In Section 3, we prove Theorem~\ref{Thm1} by using the property of influence-blocking hypergraphs to reduce the problem of estimating 
$\mathbb E[|\I(G)|]$ to a local problem on a rooted hypertree, and then using the bonus function of hypergraphs to establish a differential equation. In Section 4, we use second moment method to prove Theorem~\ref{Thm2}. In the appendix, we present the asymptotic analysis of the quantity $f(d,r)$ from Theorem \ref{Thm1}.

\section{Preliminaries}

Gamarnik and Goldberg~\cite{GG} introduce two notions for graphs, the {\em influence-blocking subgraph} and {\em bonus function}. In this section, we generalize these notions to hypergraphs and discuss their properties.
A {\em hypertree} is a linear hypergraph with no Berge cycle, and a {\em rooted hypertree} is a hypertree in which a special vertex called the {\em root} is singled out. In summary, we show that the performance of the greedy algorithm on hypergraphs with large girth is locally similar to its performance on a rooted hypertree --  note that if a hypergraph has high girth, then for each vertex, its neighbourhood within finite distance looks like a hypertree. Hence, if we can show that the event of a vertex being selected into the independent set is mostly dependent on its neighbourhood within finite distance, then we can simplify the analysis of each vertex into the analysis of the root of a rooted tree. Then we analyze the probability of the root of a rooted hypertree being selected by the randomized greedy algorithm. For ease of analysis, we consider an equivalent way to do the randomized greedy algorithm as follows:
\begin{enumerate}
    \item Equip each vertex with i.i.d.~weight from the uniform distribution on $[0,1]$. Then with probability 1,  all vertices will have distinct weights.
    \item Iteratively select the vertex with largest weight from all remaining vertices of $G$, and add it to the independent set so far, and then delete all remaining vertices that form an edge with the selected vertices, and repeat until no vertices remain.
\end{enumerate}
The strategy is to analyze the probability of each vertex being selected into the independent set. 

\subsection{Influence-blocking hypergraphs}

Garmarnik and Goldberg~\cite{GG} introduce {\em influence-blocking subgraphs}; here we extend this notion to hypergraphs. Suppose we already applied the first step of the greedy algorithm on $G$. That is, the vertices of $G$ are now equipped with distinct weights. Let $v$ be a vertex of $G$, $e$ be an edge of $G$ such that $e$ contains $v$. We say $v$ {\em defeats} $e$ if there is another vertex $v'$ in $e$ such that the weight of $v'$ is smaller than the weight of $v$. That is, $v$ is not the smallest weighted vertex in $e$. Observe that if $v$ defeats all the edges that contains it, then $v$ must be selected into $\I(G)$, since it cannot be deleted according to the rule of the algorithm. In this case, the weight of any other vertex that is not in the neighbourhood of $v$ will not influence the behaviour of $v$. This phenomenon can be generalized to sub-hypergraphs, which gives us the 
following definition:

\begin{definition}
Let $G$ be a hypergraph whose vertices are equipped with distinct weights. An induced sub-hypergraph $H$ of $G$ is called an {\em influence-blocking hypergraph} if for every vertex $v\in V(H)$, and $e\in E(G)\backslash E(H)$ with $v\in e$, $v$ is not the vertex in $e$ with smallest weight.
\end{definition}

If $G$ is a hypergraph whose vertices are already equipped with distinct weights, then we also let $\I(G)$ denote the independent set of $G$ generated by applying the second step of the greedy algorithm to $G$. Let $v$ be a vertex of $G$, such that $v\not\in\I(G)$. If $e$ is an edge of $G$, such that $v\in e$ and $e\subset v\cup\I(G)$, then we say {\em $v$ is deleted by $e$}. The first property of influence-blocking hypergraphs is that the performance of the greedy algorithm inside this sub-hypergraph is not dependent on the performance of the algorithm outside this sub-hypergraph. This phenomenon is described by the following lemma, which is a straightforward modification of Lemma $5$ in~\cite{GG}: 

\begin{lem}\label{POfIBG}
Let $G$ be a hypergraph whose vertices are equipped with distinct weights. Let $H$ be an influence-blocking hypergraph of $G$. Then $\I(H)=\I(G)\cap V(H)$.
\end{lem}

\begin{proof}
Let $V(H)=\{v_1,v_2,\dots,v_m\}$, such that $v_1>v_2>\dots>v_m$ (where $v_i>v_j$ means the weight of $v_i$ is larger than the weight of $v_j$). To prove the lemma, it suffices to show that $v_i\in\I(H)$ if and only if $v_i\in\I(G)$, for all $i$ such that $1\le i\le m$. We do that by induction. First, for $i=1$, we have $v_1\in\I(H)$. By the definition of influence-blocking hypergraph, $v_1$ cannot be deleted by edges not in $H$. Since $v_1$ has the largest weight among all vertices of $H$, so it cannot be deleted by edges in $H$ either. Hence, we also have $v_1\in\I(G)$. This completes the base case. Now suppose $1<i\le m$, and the argument holds for all integer less than  $i$. If $v_i\not\in\I(H)$, then $v_i$ must be deleted by an edge $e\in E(H)$ such that $e\backslash v_i$ consists of vertices whose weights are larger than the weight of $v_i$. Then by the inductive assumption, $v_i$ must be deleted by the same edge in the algorithm for $G$. Hence, we have $v_i\not\in\I(G)$. If $v_i\in\I(H)$, then $v_i$ cannot form an edge in $H$ with vertices whose weights are larger than the weight of $v_i$. Hence, by inductive assumption, $v_i$ cannot be deleted by edges in $H$. Also, by the definition of influence-blocking hypergraph, $v_i$ cannot be deleted by edges not in $H$ either. Therefore, we have $v_i\in\I(G)$. This completes the inductive step, and hence the proof of the lemma.
\end{proof}

The second property of the influence-blocking hypergraphs is that any subset of vertices can be extended to a unique minimal influence-blocking hypergraph, which is presented by the following lemma, which is a straightforward modification of Lemma $3$ in~\cite{GG}:

\begin{lem}
Let $G$ be a hypergraph whose vertices are equipped with distinct weights. Let $A$ be such that $A\subset V(G)$, then there exist a unique minimal influence-blocking hypergraph $\block_G(A)$ of $G$ such that $A\subset V(\block_G(A))$. It can be simplified as $\block(A)$ if there is no ambiguity.
\end{lem}

\begin{proof}
Pick a set of vertices $V_A$ as following. First, put all vertices of $A$ into $V_A$. Then, we iteratively take edges that are not in $A$ but whose smallest-weighted vertex is in $A$, and put all the vertices of such edges into $V_A$, and then repeat until no edge like this remains. Let $\block(A)$ be the sub-hypergraph of $G$ induced by $V_A$. By definition, $\block(A)$ is an influence-blocking hypergraph of $G$, and is contained in any influence-blocking hypergraph of $G$ that contains $A$. Hence, it is minimal. Also, by the process that it is generated, we can see that it is unique.
\end{proof}

\begin{definition}\label{path}
For any integers $r, l\geq1$, an {\em $(r+1)$-uniform path of length $l$} connecting $v_0$ to $v_{lr}$ is a hypergraph with vertices $\{v_0, v_1,\dots, v_{lr}\}$ and edges $e_k=\{v_{kr}, v_{kr+1},\dots, v_{(k+1)r}\}$ for ${0\leq k\leq l-1}$. If the vertices of a path are weighted and the smallest-weighted vertex in edge $e_k$ is $v_{kr}$ for all $0\leq k\leq l-1$, then we say the weighted path is {\em increasing} from $v_0$ to $v_{lr}$.
\end{definition}

Note that the definition of path here is different from the definition of a {\em Berge path}, which is defined in a similar way as the Berge cycle. 

The following lemma evaluate the probability that a path in a hypergraph is increasing when given a random total order:
\begin{lem}\label{NumOfIP}
For any integers $r, l\geq1$, the number of ways to assign $\{0,1,\dots,lr\}$ as distinct weights to the vertices of an $(r+1)$-uniform paths of length $l$ from $v_0$ to $v_{lr}$ so that it is increasing is
$$\frac{(lr+1)!}{\prod_{k=1}^l(kr+1)}.$$
Hence, for an $(r+1)$-uniform path $P$ of length $l$, if each vertex is equipped with i.i.d.~weight from the uniform distribution on $[0,1]$, then
\begin{equation}\label{PofI}
    \P[\text{$P$ is increasing from $v_0$ to $v_{lr}$} ]=\frac{1}{\prod_{k=1}^l(kr+1)}.
\end{equation}

\end{lem}
\begin{proof}
For simplicity, we only prove this for $r=2$. In this case, we want to show that the number of proper weight assignments for paths of length $l$ is $(2l)!!=\prod_{k=1}^l(2k)$. The idea of the proof for general case is exactly the same. Let $a_l$ be the number of proper weight assignments for $3$-uniform paths of length $l$ with distinct weights from $\{0,1,\dots,2l\}$. Let $W_i$ be the weight of $v_i$. We prove $a_l=(2l)!!$ by induction. First, for $l=1$, $W_0$ has to be $0$, $W_{2}$ can be either $1$ or $2$. So $a_1=2=2!!$. Now for $l\geq 2$, suppose the lemma is true for $l-1$. Then again, $W_{0}$ has to be $0$. $W_2$ is less than all $W_i$ with $i>2$, so $W_2$ is at least the third smallest weight. As a result, $W_{2}=1$ or $2$. When $W_{2}=1$, $W_1$ can be any number in $\{2,3,\dots,2l\}$, and all the other vertices form a $3$-uniform increasing path of length $l-1$. So the number of proper weight assignments of this kind is $(2l-1)a_{l-1}$. When $W_{2}=2$, $W_{1}$ has to be 1, and all the other vertices form a $3$-uniform increasing path of length $l-1$, the number of proper weight assignments of this kind is $a_{l-1}$. Hence, by inductive assumption, we have $a_l=2la_{l-1}=2l\cdot(2l-2)!!=(2l)!!$. This completes the proof for $r=2$.
\end{proof}

For any vertex $v$ and any integer $h\geq 1$, let $N_h(v)$ be the set of vertices $w$ such that there exist a path, as defined in Definition~\ref{path},  connecting $v$ to $w$, whose length is less or equal than $h$. When $h=0$, let $N_0(v)={v}$. The following lemma, which is a modification of Lemma $6$ in~\cite{GG}, show that for any vertex $v$, the probability that the minimal influence-blocking hypergraph containing $v$ is not a sub-hypergraph of $N_h(v)$ converges to $0$ as $h\rightarrow\infty$.

\begin{lem}\label{ProbOfE}
For any integers $r\geq1$ and $d\geq2$, let $G$ be any $(r+1)$-uniform linear hypergraph of maximum degree $d$, and suppose that the vertices are equipped with i.i.d.~weights from the uniform distribution on $[0,1]$. Then for any vertex $v$ and any $h\geq 0$,
$$\P[\block(v)\not\subset N_h(v)]\leq \frac{d(d-1)^h}{r\prod_{k=1}^{h+1}(k+\frac{1}{r})}.$$
\end{lem}
\begin{proof}
For any vertex $v$ there exist at most $d(d-1)^hr^h$ distinct paths of length $h+1$ that connecting $v$ to some vertex in $N_{h+1}(i)\backslash N_h(i)$. By definition, $\block(v)\not\subset N_h(v)$ if and only if at least one of these path is increasing. So by applying a union bound and equation~(\ref{PofI}), we have

$$
\P[\block(v)\not\subset N_h(v)]
\le\frac{d(d-1)^hr^h}{\prod_{k=1}^{h+1}(kr+1)}
= \frac{d(d-1)^h}{r\prod_{k=1}^{h+1}(k+\frac{1}{r})}.
$$
\end{proof}

\subsection{Bonus function of hypergraphs}

To analyze the probability of the root of a rooted hypertree being selected into the independent set, we use the following notion to establish a recursive equation, and hence by some analysis, a differential equation. 

Consider the following {\em bonus function of hypergraphs}, which is extended from the bonus function of graphs introduced by Garmarnik and Goldberg~\cite{GG}:

\begin{definition}
Let $T$ be a rooted hypertree, whose vertices are equipped with distinct positive weights. Let $W_v$ be the weight of a vertex $v$, $DE(v)$ be the set of descending edges of $v$ and $I$ be the indicator function.
Then the {\em bonus function of hypergraphs} $S_T:V(T)\rightarrow\R$ is defined by
$$
S_T(v)=
\left\{
    \begin{aligned}
    &W_v, &&\text{v is leaf,}\\
    &W_v\prod_{e\in DE(v)}I(W_v>\min_{u\in e,u\not=v}\{S_T(u)\}),\ &&\text{otherwise.}
    \end{aligned}
\right.
$$
\end{definition}

Given a weighted rooted tree, the bonus function value of the root is exactly the weight of the root if the root is selected by the greedy algorithm, and is $0$ if the the root is not selected, as shown by the following lemma:

\begin{lem}\label{POfBF}
Let $T$ be a rooted hypertree, whose vertices are equipped with distinct positive weights. Let $\gamma$ be the root of $T$, $W_\gamma$ be the weight of $\gamma$, then we have
$$S_T(\gamma)=W_{\gamma}I(\gamma\in\I(T)).$$
\end{lem}

\begin{proof}
We prove by induction on the height of the tree. When the height is $0$, this lemma is true. Now suppose $T$ has height $h>0$, and this lemma holds for all trees with height less than $h$. Let $e_k,\ 1\leq k \leq d$, be all descending edges of the root $\gamma$. Then by definition of the bonus function, we have
$$S_T(\gamma)=W_\gamma\prod_{k=1}^{d}I(W_\gamma>\min_{v\in e_k, v\not=\gamma}\{S_T(v)\}).$$
So it suffices to show that
$$\prod_{k=1}^{d}I(W_\gamma>\min_{v\in e_k, v\not=\gamma}\{S_T(v)\})=I(\gamma\in\I(T)).$$
Let $T_{v}$ be the subtree of $T$ with root $v$, such that $T_v$ contains only the edges descending from $v$.
If $W_\gamma>\min_{v\in e_k, v\not=\gamma}\{S_T(v)\}$ for all $k$ such that $1\leq k\leq d$. For an arbitrary $k$, pick $v\in e_k$, $v\not=\gamma$, such that $W_\gamma>S_T(v)$, then there are two cases. Firstly, if $W_\gamma<W_{v}$, then we have $S_T(v)=0$. By inductive assumption, this implies $v\not\in\I(T_{v})$. Then by Lemma \ref{POfIBG}, since $T_v$ is an influence-blocking hypergraph of $T$, we have $v\not\in\I(T)$. This means that $\gamma$ will not be deleted by $e_k$. Secondly, if $W_\gamma>W_{v}$. This also means that $\gamma$ will not be deleted by $e_k$. This argument works for all $1\leq k\leq d$. Therefore, $\gamma\in\I(T)$.

On the other hand, if $W_\gamma<\min_{v\in e_k, v\not=\gamma}\{S_T(v)\}$ for some $k$, then $W_\gamma$ must be the smallest-weighted vertex in $e_k$ and $v\in\I(T_v)$ for all $v\in e_k$, $v\not=\gamma$. Since $T_v$ is an influence-blocking hypergraph of $T$, by Lemma \ref{POfIBG} we have $v\in\I(T)$ for all $v\in e_k$, $v\not=\gamma$. This implies that $\gamma$ will be deleted by $e_k$. Therefore, $\gamma\not\in\I(T)$.
\end{proof}

Let $T(d,h)$ be the $(r+1)$-uniform rooted hypertree such that all non-leaf vertices have $d$ descending edges, and all leaves have depth $h$. Let $\tilde{T}(d,h)$ be the $(r+1)$-uniform rooted hypertree such that the root has $d$ descending edges while all other non-leaf vertices have $d-1$ descending edges, and all leaves have depth $h$.

Let $\gamma$ be the root of $T(d,h)$. Apply the first step of the greedy algorithm to $T(d,h)$, that is,  randomly assign weights to $T(d,h)$. Let $F_{d,h}$ be the distribution function of $S_{T(d,h)}(\gamma)$. That is, $F_{d,h}(x)=\P[S_{T(d,h)}(\gamma)\leq x]$. Similarly, let $\tilde{F}_{d,h}$ be the distribution function of $S_{\tilde{T}(d,h)}(\gamma)$. That is, $\tilde{F}_{d,h}(x)=\P[S_{\tilde{T}(d,h)}(\gamma)\leq x]$. Note that by Lemma \ref{POfBF}, we have
\begin{align}
    &1-F_{d,h}(0)=\P[\gamma\in\I(T(d,h)]\\
    &1-\tilde{F}_{d,h}(0)=\P[\gamma\in\I(\tilde{T}(d,h)]
\end{align}
Also by definition of the bonus function of hypergraphs, $F_{d,h}$ and $\tilde{F}_{d,h}$ satisfy the following recursive equations for all $x\in[0,1]$:
\begin{equation}\label{FdRec}
    F_{d,h}(x)=1-\int^1_x\P[S_{T(d,h)}(\gamma)=W_\gamma|W_\gamma=t]dt=1-\int^1_x[1-(1-F_{d,h-1}(t))^{r}]^d dt
\end{equation}
\begin{equation}\label{tFdRec}
    \tilde{F}_{d,h}(x)=1-\int^1_x\P[S_{\tilde{T}(d,h)}(\gamma)=W_\gamma|W_\gamma=t]dt=1-\int^1_x[1-(1-F_{d-1,h-1}(t))^{r}]^d dt
\end{equation}
In order to get a differential equation, we need to show that $F_{d,h}$ and $\tilde{F}_{d,h}$ converge as $h\rightarrow\infty$. We make use of the following lemma:
\begin{lem}\label{IneqOfF}
For any $x\in\R$ and integer $h\geq0$, the following inequalities hold:
\begin{equation}\label{Os1}
    (-1)^hF_{d,h}(x)\leq(-1)^hF_{d,h+1}(x)
\end{equation}
\begin{equation}\label{Os2}
    (-1)^hF_{d,h}(x)\leq(-1)^hF_{d,h+2}(x)
\end{equation}

\end{lem}
\begin{proof}
We prove inequality (\ref{Os1}) by induction. First, when $h=0$, by definition we have $F_{d,0}(x)\leq F_{d,1}(x)$. Now for $h\geq1$, suppose inequality (\ref{Os1}) holds for $h-1$. Replace $h$ by $h+1$ in equality (\ref{FdRec}) and consider its difference with the original equality, we have
$$F_{d,h+1}(x)-F_{d,h}(x)=\int^1_x\l(\l(1-(1-F_{d,h-1}(t))^{r}\r)^d-\l(1-(1-F_{d,h}(t))^{r}\r)^d\r)dt$$
Using this equation, we can check that when $F_{d,h}(x)\geq F_{d,h-1}(x)$, we have $F_{d,h+1}(x)\leq F_{d,h}(x)$; and when $F_{d,h}(x)\leq F_{d,h-1}(x)$, we have $F_{d,h+1}(x)\geq F_{d,h}(x)$. Hence, by inductive assumption, we have $(-1)^hF_{d,h}(x)\leq(-1)^hF_{d,h+1}(x)$. This completes the proof for inequality (\ref{Os1}). Same reasoning gives the proof for inequality (\ref{Os2}).
\end{proof}

\begin{cor}\label{CofF}
There exist functions $F_{d,even}(x):\R\rightarrow[0,1]$ and $F_{d,odd}(x):\R\rightarrow[0,1]$ such that the sequence of functions $\{F_{d,2k}(x)\}_{k\ge0}$ converges pointwise to $F_{d,even}(x)$ and the sequence of functions $\{F_{d,2k+1}(x)\}_{k\ge0}$ converges pointwise to $F_{d,odd}(x)$, and $F_{d,even}(x)\le F_{d,odd}(x)$ for all $x\in\R$.
\end{cor}
\begin{proof}
    As a result of inequality~(\ref{Os2}), for any $x\in\R$, the sequence $\{F_{d,2k}(x)\}_{k\geq 0}$ is increasing and the sequence $\{F_{d,2k+1}(x)\}_{k\geq 0}$ is decreasing. Also, by inequality~(\ref{Os1}), both sequences are bounded. Hence, by the Monotone Convergence Theorem~\cite{R}, they must converge, which implies the existence of $F_{d,even}(x)$ and $F_{d,odd}(x)$. The inequality can be obtained by considering the inequality~(\ref{Os1}) with $h=2k$ and $k\rightarrow\infty$.
\end{proof}
Similar results as Lemma~\ref{IneqOfF} and Corollary~\ref{CofF} for $\tilde{F}_{d,h}$ can also be obtained using the same idea, and we omit the details.

\section{Proof of Theorem~\ref{Thm1}}
The following lemma, which is a modification of Theorem $7$ in~\cite{GG}, provide an upper bound for the difference between the probability that a vertex $v$ in a hypergraph $G$ is selected and the probability that the root $\gamma$ of a rooted hypertree is selected by the greedy algorithm, showing that the performance of the greedy algorithm on $G$ is locally similar to that on a hypertree.
\begin{lem}\label{LocalProb}
	For any integers $r\geq1$, $d\geq2$ and $g\geq4$, let $G$ be an $(r+1)$-uniform $d$-regular hypergraph with girth $g$. Let $h_0=\lfloor \frac{g-3}{2}\rfloor$,  $T=\tilde{T}(d,h)$ with $h\geq h_0+1$, let $\gamma$ be the root of $T$. Then for every vertex $v\in V(G)$,
	\begin{equation}\label{LP}
	    |\P[v\in \I(G)]-\P[\gamma\in \I(T)]|\leq \frac{d(d-1)^{h_0}}{r\prod_{k=1}^{h_0+1}(k+\frac{1}{r})}.
	\end{equation}
\end{lem}
\begin{proof}
	We apply the first step of greedy algorithm on $G$ and $T$ in the following way. We first give vertices of $G$ i.i.d.~weights from the uniform distribution on $[0,1]$. Observe that $N_{h_0+1}(v)$ is a $\tilde{T}(d,h_0+1)$ hypertree, so we can find an isomorphism $f$ that maps $N_{h_0+1}(v)$ to $N_{h_0+1}(\gamma)$. Then we give the vertices in $N_{h_0+1}(\gamma)$ the same weight as their coimage in $N_{h_0+1}(v)$. Finally we give all remaining vertices in $T$ i.i.d.~weights from the uniform distribution on $[0,1]$. Then we apply the second step of greedy algorithm on both $G$ and $T$ to get $\I(G)$ and $\I(T)$. In this setting, we have the following estimate:
	\begin{align*}
	\P[v\in\I(G)]
	=&\P[v\in\I(G),\block_G(v)\subset N_{h_0}(v)]+\P[v\in\I(G),\block_G(v)\not\subset N_{h_0}(v)]\\
	=&\P[\gamma\in\I(T),\block_{T}(\gamma)\subset N_{h_0}(\gamma)]+\P[v\in\I(G),\block_G(v)\not\subset N_{h_0}(v)]\tag{\text{Lemma~\ref{POfIBG}}}\\
	\leq&\P[\gamma\in\I(T)]+\P[\block_G(v)\not\subset N_{h_0}(v)].
	\end{align*}
	This implies that
	\begin{align*}
	\P[v\in\I(G)]-\P[\gamma\in\I(T)]&\leq\P[\block_G(v)\not\subset N_{h_0}(v)]\\
	&\leq \frac{d(d-1)^{h_0}}{r\prod_{k=1}^{h_0+1}(k+\frac{1}{r})}\tag{\text{Lemma~\ref{ProbOfE}}}.
	\end{align*}
	 We complete the proof by repeating the reasoning above with the roles of $\P[v\in\I(G)]$ and $\P[\gamma\in\I(T)]$ reversed.	
\end{proof}

Using similar idea as in the proof above, we can also show that the following limits exist:
\begin{lem}\label{limexist}
	For any fixed integer $d$, the limits
	$\lim_{h\rightarrow\infty}{\P[\gamma\in \I({T}(d,h))]}$ and $\lim_{h\rightarrow\infty}{\P[\gamma\in \I(\tilde{T}(d,h))]}$
	exist, where $\gamma$ denote the root of the rooted hypertrees.
\end{lem}

\begin{proof}
	 We only present the proof of the existence of $\lim_{h\rightarrow\infty}{\P[\gamma\in \I(\tilde{T}(d,h))]}$. The proof of the existence of $\lim_{h\rightarrow\infty}{\P[\gamma\in \I({T}(d,h))]}$ is similar and we omit the details. Let $h$, $h'$ be positive integers with $h'>h$. Using the same idea as in the proof of Lemma~\ref{LocalProb}, we can show that
	$$
	|\P[\gamma\in \I(\tilde{T}(d,h))]-\P[\gamma\in \I(\tilde{T}(d,h'))]|\leq \frac{d(d-1)^{h-1}}{r\prod_{k=1}^{h}(k+\frac{1}{r})}\rightarrow0\ \text{as $h\rightarrow\infty$}.
	$$
	So we conclude that the sequence $\{\P[\gamma\in\I(\tilde{T}(d,h))]\}_{h\geq 1}$ is a Cauchy sequence and therefore has a limit.
\end{proof}

Now we are ready to show that $F_{d,h}(x)$ and $\tilde{F}_{d,h}(x)$ converge, and hence get the differential equations we need:
\begin{lem}\label{DefEq}
	there exist functions $F_{d}(x)$ and $\tilde{F}_{d}(x)$such that $F_{d,h}(x)$ converges pointwise to $F_{d}(x)$ and $\tilde{F}_{d,h}(x)$ converges pointwise to $\tilde{F}_{d}(x)$ as $h\rightarrow\infty$. $F_{d}(x)$ and $\tilde{F}_{d}(x)$ satisfy the following equations:
	\begin{equation}\label{FdE}
		F_{d}(x)=1-\int^1_x[1-(1-F_{d}(t))^{r}]^d dt,
	\end{equation}
	\begin{equation}\label{tFdE}
		\tilde{F}_{d}(x)=1-\int^1_x[1-(1-F_{d-1}(t))^{r}]^d dt.
	\end{equation}	
\end{lem}

\begin{proof}
	We only present the proof of the existence of $F_d$ here. The proof of the existence of $\tilde{F}_{d}$ is similar and we omit the details. By Corollary \ref{CofF}, there exist $F_{d,even}(x)$ and $F_{d,odd}(x)$ such that $F_{d,2k}(x)$ converges pointwise to $F_{d,even}(x)$ and $F_{d,2k+1}(x)$ converges pointwise to $F_{d,odd}(x)$ as $k\rightarrow\infty$. Hence, to prove the existence of $F_d$, it suffices to show that $F_{d,even}(x)=F_{d,odd}(x)$ for all $x\in\R$. By Lemma \ref{limexist}, $\lim_{h\rightarrow\infty}{\P[\gamma\in\I(T(d,h))]}$ exists. Since $F_{d,h}(0)=1-\P[\gamma\in\I(T(d,h))]$, this implies that $\lim_{h\rightarrow\infty}{F_{d,h}(0)}$ exists. So we have
	$$F_{d,even}(0)=\lim_{h\rightarrow\infty}{F_{d,h}(0)}=F_{d,odd}(0).$$
	Now consider equation~(\ref{FdRec}) with $h=2k$, and let $k$ go to infinity on both sides, and then use the Dominated Convergence Theorem~\cite{R}, we have
	$$F_{d,even}(x)=1-\int^1_x[1-(1-F_{d,odd}(t))^{r}]^d dt.$$
	Similarly, we also have
	$$F_{d,odd}(x)=1-\int^1_x[1-(1-F_{d,even}(t))^{r}]^d dt.$$
	Take the derivative on both sides and then take the difference of these two equations, we have
	$$F'_{d,even}(x)-F'_{d,odd}(x)=[1-(1-F_{d,odd}(x))^{r}]^d-[1-(1-F_{d,even}(x))^{r}]^d\ge 0,$$
	where the inequality comes from the fact that $F_{d,even}\leq F_{d,odd}$ by Corollary~\ref{CofF}. So for any fixed $x\in[0,1]$,
	$$F_{d,even}(x)=F_{d,even}(0)+\int^x_0F'_{d,even}(t)dt\geq F_{d,odd}(0)+\int^x_0F'_{d,odd}(t)dt=F_{d,odd}(x).$$
	This combined with the inequality $F_{d,even}\leq F_{d,odd}$, implies $F_{d,even}=F_{d,odd}$. This completes the proof of the existence of $F_d(x)$. Now consider equations~(\ref{FdRec}) and~(\ref{tFdRec}), let $h\rightarrow\infty$ and then use the Dominated Convergence Theorem~\cite{R}, we get the desired differential equations.
\end{proof}

\begin{lem}\label{Fd}
	For any integer $d\ge 3$, let $G_d(x)=1-F_{d-1}(x)$, then $G_d(x)$ satisfies the following equation:
    \begin{equation}
    	1-\sum_{n\geq0} \binom{n+d-2}{d-2}\frac{G_d(x)^{rn+1}}{rn+1}=x.
    \end{equation}
\end{lem}

\begin{proof}
	By equation~(\ref{FdE}), we have
	$$
	G_d(x)=\int_{x}^{1}(1-G_d(t)^r)^{d-1}dt.
	$$
	Taking derivatives on both sides, we have
	$$
	G_d'(x)=-(1-G_d(x)^{r})^{d-1}.
	$$
	Let $H_d(x)=\sum_{n\geq0} \binom{n+d-2}{d-2}\frac{x^{rn+1}}{rn+1}$, it is not hard to check that $H_d'(x)=\frac{1}{(1-x^r)^{d-1}}$. So the equation above is equivalent to
	$$
	\l(H_d\l(G_d(x)\r)\r)'=-1.
	$$
	Solving this equation, we obtain
	$$\sum_{n\geq0} \binom{n+d-2}{d-2}\frac{G_d(x)^{rn+1}}{rn+1}=-x+C$$
	Let $x=1$, we have $0=-1+C$, which implies $C=1$. This completes the proof.
\end{proof}

\begin{lem}\label{tFd}
    For any integer $d\ge 3$, let $\tilde{G}_d(x)=1-\tilde{F}_d(x)$, then we have the following equation:
    \begin{equation}
        \tilde{G}_d(x)=G_d(x)-\frac{G_d(x)^{r+1}}{r+1}
    \end{equation}
\end{lem}

\begin{proof}
    By equation~(\ref{tFdE}),
    \begin{equation*}
        \tilde{G}_d(x)=\int^1_x(1-G_d(t)^r)^d dt
    \end{equation*}
    Consider changing the variable in the integral by letting $u=G_d(t)$. By equation~(\ref{FdE}), not hard to see $dt=-\frac{du}{(1-u^r)^{d-1}}$. Hence,
    \begin{equation*}
        \tilde{G}_d(x)=-\int^{G_d(1)}_{G_d(x)}(1-u^r)du=G_d(x)-\frac{G_d(x)^{r+1}}{r+1}
    \end{equation*}
\end{proof}

Now we are ready to prove Theorem 4.

\begin{proof}[Proof of Theorem~\ref{Thm1}]
	Applying inequality~(\ref{LP}), we have
	\begin{align*}
	|\frac{\E[|\I(G)|]}{n}-\P[\gamma\in\I(\tilde{T}(d,h))]|
	&\leq\frac{1}{n}\sum_{v\in V(G)} |\P[v\in \I(G)]-\P[\gamma\in \I(\tilde{T}(d,h))]|\\
	&\leq\frac{d(d-1)^{h_0}}{r\prod_{k=1}^{h_0+1}(k+\frac{1}{r})}
	\end{align*}
	Note that this inequality holds for all $h\geq h_0+1$. Let $h\rightarrow\infty$, we have
	$$
    |\frac{\E[|\I(G)|]}{n}-\lim_{h\rightarrow\infty}{\P[\gamma\in \I(\tilde{T}(d,h))]}|\leq\frac{d(d-1)^{h_0}}{r\prod_{k=1}^{h_0+1}(k+\frac{1}{r})}.
	$$
	Let $f(d,r)=\lim_{h\rightarrow\infty}{\P[\gamma\in \I(\tilde{T}(d,h))]}=\tilde{G}_d(0)$, then we have the required inequality~(\ref{main}). Let $u(d,r)=\lim_{h\rightarrow\infty}{\P[\gamma\in \I({T}(d-1,h))]}=G_d(0)$.  By Lemma~\ref{Fd}, we know that $u(d,r)$ satisfy equation~(\ref{ud}). By Lemma~\ref{tFd}, we have
	\begin{equation*}
	    f(d,r)=\tilde{G}_d(0)=G_d(0)-\frac{G_d(0)^{r+1}}{r+1}=u(d,r)-\frac{u(d,r)^{r+1}}{r+1}.
	\end{equation*}
	This completes the proof.
\end{proof}

\section{Proof of Theorem~\ref{Thm2}}

In section 2, we notice that vertices that are far away from each other are very likely ``independent''. More accurately, if two vertices $u$,$v$ are far away from each other, then the indicator of the event that $u$ is selected and the indicator of the event that $v$ is selected by the greedy algorithm have small covariance. This phenomenon can also be used to give an upper bound for the variance of the algorithm.

\begin{lem}\label{Var}
	For any integers $r\geq1$ and $d\geq2$, let $G$ be an $(r+1)$-uniform linear hypergraph on $n$ vertices with maximum degree $d$, then the variance satisfies:
	\begin{equation}
	    \mathrm{Var}[{\I(G)}]\leq {3d^2r^2e^{r^2(d-1)^3}}n.
	\end{equation}
\end{lem}

\begin{proof}
Let $V(G)=\{v_1, v_2,\dots,v_n\}$, $X_i=I(v_i\in\I(G))$. Then
\begin{align*}
    \mathrm{Var}(\I(G))&=\mathrm{Var}(\sum_{i=1}^n X_i)\\
    &=\sum_{i=1}^n(\E[X_i^2]-\E[X_i]^2)+\sum_{1\leq i\not=j\leq n}(\E[X_iX_j]-\E[X_i]\E[X_j])\\
    &\leq n+\sum_{1\leq i\leq n}\sum_{\delta\geq1}\sum_{v_j\in N_{\delta}(v_i)\backslash N_{\delta-1}(v_i)}(\E[X_iX_j]-\E[X_i]\E[X_j]),
\end{align*}
where the inequality uses the bound $(\E[X_i^2]-\E[X_i]^2)\leq1$. For any $1\leq i\leq n$, we consider the sum
$$\sum_{\delta\geq1}\sum_{v_j\in N_{\delta}(v_i)\backslash N_{\delta-1}(v_i)}(\E[X_iX_j]-\E[X_i]\E[X_j])$$\\
First, for any $\delta\geq3$, let $h=\lfloor\frac{\delta-3}{2}\rfloor$, and let $A_{i,h}$ denote the event $\{\block(v_i)\not\subset N_h(v_i)\}$, $A_{i,h}^c$ denote the complement of the event $A_{i,h}$, that is $\{\block(v_i)\subset N_h(v_i)\}$. This event is only determined by the weights of vertices in $N_{h+1}(v_i)$. Notice that for every $v_j\in N_{\delta}(v_i)\backslash N_{\delta-1}(v_i)$, $N_{h+1}(v_i)\cap N_{h+1}(v_j)=\emptyset$. So $A_{i,h}^c$ and $A_{j,h}^c$ are independent. Then we have,
\begin{align*}
    \E[X_iX_j]&=\P[v_i\in\I(G), v_j\in\I(G)]\\
    &=\P[v_i\in\I(G), v_j\in\I(G), A_{i,h}^c\cap A_{j,h}^c]+\P[v_i\in\I(G), v_j\in\I(G), A_{i,h}\cup A_{j,h}].
\end{align*}
By Lemma~\ref{POfIBG} and the independence between $A_{i,h}^c$ and $A_{j,h}^c$, we have
\begin{align*}
    \P[v_i\in\I(G), v_j\in\I(G), A_{i,h}^c\cap A_{j,h}^c]&=\P[v_i\in\I(\block(v_i)), v_j\in\I(\block(v_j)), A_{i,h}^c\cap A_{j,h}^c]\\
    &=\P[v_i\in\I(\block(v_i)), A_{i,h}^c]\P[v_j\in\I(\block(v_j)), A_{j,h}^c]\\
    &\le\E[X_i]\E[X_j].
\end{align*}
On the other hand, by Lemma~\ref{ProbOfE}
\begin{align*}
    \P[v_i\in\I(G), v_j\in\I(G), A_{i,h}\cup A_{j,h}]&\le\P[A_{i,h}]+\P[A_{j,h}]\\
    &\le\frac{2d(d-1)^h}{r\prod_{k=1}^{h+1}(k+\frac{1}{r})}
\end{align*}
Hence,
$$
\E[X_iX_j]-\E[X_i]\E[X_j]\le\frac{2d(d-1)^h}{r\prod_{k=1}^{h+1}(k+\frac{1}{r})}
$$
Since $G$ has maximum degree $d$, we have $|N_{\delta}(v_i)\backslash N_{\delta-1}(v_i)|\leq d(d-1)^{\delta-1}r^{\delta}$. \\
In particular, for odd integer $\delta\geq3$, we have $\delta=2h+3$. So the sum
\begin{align*}
    \sum_{odd\ \delta\geq3}\sum_{v_j\in N_{\delta}(v_i)\backslash N_{\delta-1}(v_i)}(\E[X_iX_j]-\E[X_i]\E[X_j])
    &\leq\sum_{h\geq0}d(d-1)^{2h+2}r^{2h+3}\frac{2d(d-1)^h}{r(h+1)!}\\
    &=\frac{2d^2}{d-1}\sum_{h\geq1}\frac{r^{2h}(d-1)^{3h}}{h!}\\
    &\leq2d^2\sum_{h\geq1}\frac{r^{2h}(d-1)^{3h}}{h!}
\end{align*}
For even integer $\delta\geq3$, we have $\delta=2h+4$. So the sum
\begin{align*}
    \sum_{even\ \delta\geq3}\sum_{v_j\in N_{\delta}(v_i)\backslash N_{\delta-1}(v_i)}(\E[X_iX_j]-\E[X_i]\E[X_j])
    &\leq\sum_{h\geq0}d(d-1)^{2h+3}r^{2h+4}\frac{2d(d-1)^h}{r(h+1)!}\\
    &=2d^2r\sum_{h\geq1}\frac{r^{2h}(d-1)^{3h}}{h!}\\
\end{align*}
For $1\leq\delta\leq2$, use the bound $\E[X_iX_j]-\E[X_i]\E[X_j]\leq 1$, we have
$$\sum_{1\leq\delta\leq2}\sum_{v_j\in N_{\delta}(v_i)\backslash N_{\delta-1}(v_i)}(\E[X_iX_j]-\E[X_i]\E[X_j])\leq 2d^2r^2
$$
Combine the three inequalities above, we have
$$\sum_{\delta\geq1}\sum_{v_j\in N_{\delta}(v_i)\backslash N_{\delta-1}(v_i)}(\E[X_iX_j]-\E[X_i]\E[X_j])\leq2d^2r^2e^{r^2(d-1)^3}$$
So the variance
$$\mathrm{Var}(\I(G))\leq n+2d^2r^2e^{r^2(d-1)^3}n\leq3d^2r^2e^{r^2(d-1)^3}n$$
\end{proof}

\begin{proof}[Proof of Theorem \ref{Thm2}]
	By Lemma~\ref{Var}, we know that for fix $d$ and $r$, there exist a constant $c$ such that $\mathrm{Var}(\I(G))<cn$. Hence, by Chebyshev's Inequality we have
	$$
	\P[||\I(G)|-\E[|\I(G)|]|>\sqrt{n}b(n)]\le\frac{\mathrm{Var}(|\I(G)|)}{b(n)^2n}\rightarrow0,\ \text{as $n\rightarrow\infty$}.$$
\end{proof}

\section*{Appendix A}

We first collect some real number inequalities:

\begin{prop}\label{EandI}
Let $n,r,d$ be positive integers. Then 
\begin{center}
\begin{tabular}{lp{4in}}
   $1$. &  For $x \geq 0$, 
   \begin{equation}\label{E1}
        \int_0^xe^{t^{r}}dt=\sum_{n\geq 0}\frac{x^{rn+1}}{n!(rn+1)}
    \end{equation} \\
    $2$. & For $n\le\sqrt{d}$, 
    \begin{equation}\label{I1}
        \Bigl(1+\frac{n}{d}\Bigr)^n<e^{\frac{n^2}{d}}<e
    \end{equation} \\
$3$. & For $y \geq 0$,   
    \begin{equation}\label{I2}
        \Bigl(\frac{y}{n}\Bigr)^n\leq e^{\frac{y}{e}}
    \end{equation}
\end{tabular}
\end{center}
\end{prop}

Let $u_d=\lim_{h\rightarrow\infty}\P[\gamma\in\I(T(d,h))]=u(d+1,r)$. Note that $u_d$ can be viewed as the probability of the root of $T(d,\infty)$ being selected by the greedy algorithm, while $f(d,r)$ can be viewed as the probability of the root of $\tilde{T}(d,\infty)$ being selected by the greedy algorithm.
\begin{prop}\label{AsymAnal}
	$
	f(d,r)\sim(\frac{\log d}{rd})^{\frac{1}{r}}
	$
	as $d\rightarrow\infty$.
\end{prop}

\begin{proof}
	Let $g(d,u)=\sum_{n\geq 0}\binom{n+d-1}{n}\frac{u^{rn+1}}{rn+1}$. It is not hard to see that $g$ is increasing with respect to $u$. By Lemma~\ref{Fd}, we have $g(d,u_d)=1$. Now for any $\epsilon>0$, let $u=((\frac{1}{r})^{\frac{1}{r}}+\epsilon)(\frac{\log d}{d})^{\frac{1}{r}}$, we have
	\setlength{\jot}{12pt}
	\begin{align*}
	g(d,u)&\geq\sum_{n\geq 0}\frac{d^n}{n!}\frac{u^{rn+1}}{rn+1}\\
	&=d^{-\frac{1}{r}}\sum_{n\geq 0}\frac{(ud^{\frac{1}{r}})^{rn+1}}{n!(rn+1)}\\
	&=d^{-\frac{1}{r}}\int_0^{ud^{\frac{1}{r}}}e^{t^{r}}dt \tag{by~\ref{E1}}\\
	&>d^{-\frac{1}{r}}\int_{(\log d)^{\frac{1}{r}}(\frac{1}{r})^{\frac{1}{r}}}^{(\log d)^{\frac{1}{r}}[(\frac{1}{r})^{\frac{1}{r}}+\epsilon]}e^{t^{r}}dt\\
	&>d^{-\frac{1}{r}}(\epsilon(\log d)^{\frac{1}{r}})e^{\frac{\log d}{r}}\\
	&=\epsilon(\log d)^{\frac{1}{r}}.
	\end{align*}

	This means that $g(d,u)\rightarrow\infty$ as $d\rightarrow\infty$, hence $u_d<[(\frac{1}{r})^{\frac{1}{r}}+\epsilon](\frac{\log d}{d})^{\frac{1}{r}}$ when $d$ is large enough.\\
	On the other hand, for any $\epsilon>0$, let $u=c(\frac{\log d}{d})^{\frac{1}{r}}$, where $c=(\frac{1}{r}-\epsilon)^{\frac{1}{r}}$, we have
	\begin{align*}
	g(d,u)&\leq\sum_{n\geq 0}\l(\frac{e(n+d)}{n}\r)^n\frac{u^{rn+1}}{rn+1}\\
	&=\sum_{n\geq 0}\frac{u}{rn+1}\l(\frac{e}{n}+\frac{e}{d}\r)^n(c^{r}\log d)^n
	\end{align*}
	When $n\geq 4c^{r}e\log d$, and $d$ is large enough, we have
	\begin{align*}
	\setlength{\jot}{12pt}
	\sum_{n\geq 4c^{r}e\log d}\frac{u}{rn+1}\l(\frac{e}{n}+\frac{e}{d}\r)^n(c^{r}\log d)^n
	&\leq u\sum_{n\geq 4c^{r}e\log d}\l(\frac{2e}{4c^{r}e\log d}\r)^n(c^{r}\log d)^n\\
	&=u\sum_{n\geq 4c^{r}e\log d}\l(\frac{1}{2}\r)^n\\
	&<c\l(\frac{\log d}{d}\r)^{\frac{1}{r}}\rightarrow 0\ as\ d\rightarrow\infty.
	\end{align*}
	When $n<4c^{r}e\log d$, and $d$ is large enough, we have
	
	\begin{align*}
	\sum_{n< 4c^{r}e\log d}\frac{u}{rn+1}\l(\frac{e}{n}+\frac{e}{d}\r)^n(c^{r}\log d)^n
	&<u\sum_{n< 4c^{r}e\log d}\l(1+\frac{n}{d}\r)^n\l(\frac{c^{r}e\log d}{n}\r)^n\\
	&<ue\sum_{n< 4c^{r}e\log d}\l(\frac{c^{r}e\log d}{n}\r)^n \tag{by~\ref{I1}}\\
	&<c\l(\frac{\log d}{d}\r)^{\frac{1}{r}}e(4c^{r}e\log d)e^{c^{r}\log d} \tag{by~\ref{I2}}\\
	&=4e^2c^{r+1}(\log d)^{\frac{r+1}{r}}d^{-\epsilon}\rightarrow 0\ as\ d\rightarrow\infty.
	\end{align*}
	This means that $g(d,u)\rightarrow 0$ as $d\rightarrow\infty$, hence $u_d>(\frac{1}{r}-\epsilon)^{\frac{1}{r}}(\frac{\log d}{d})^{\frac{1}{r}}$ when $d$ is large enough.\\
	These estimates imply that $u_d\sim(\frac{\log d}{rd})^{\frac{1}{r}}$, hence $u_d\rightarrow 0$ as $d\rightarrow\infty$. Recall that by Theorem~\ref{Thm1},
	\begin{equation*}
	    f(d,r)=u(d,r)-\frac{u(d,r)^{r+1}}{r+1}=u_{d-1}-\frac{u_{d-1}^{r+1}}{r+1}
	\end{equation*}
	
	Therefore, $f(d,r)\sim(\frac{\log d}{rd})^{\frac{1}{r}}$ as $d\rightarrow\infty$.
\end{proof}

\section{Acknowledgements}
We would like to thank Patrick Bennett and Deepak Bal for some helpful comments on simplifying the formulation of the main theorem. We would
also like to thank the anonymous referees for their careful reading of the paper and useful suggestions.

\end{document}